\documentclass[10pt]{article}

\usepackage{amssymb}
\usepackage{amsthm}
\usepackage{graphicx}
\usepackage[shortcuts]{extdash}
\usepackage[hidelinks]{hyperref}

\newtheorem{theorem}{Theorem}
\newtheorem{lemma}{Lemma}

\newtheorem{remark}{Remark}

\newtheorem{conjecture}{Conjecture}

\newcommand{\beq}{\begin{equation}}
\newcommand{\eeq}{\end{equation}}
\def\Frac#1#2{\frac{\displaystyle{#1}}{\displaystyle{#2}}}

\renewcommand{\thefootnote}{\arabic{footnote}}

 \title{Optimized bounds for the product and the ratios of modified Bessel functions}
\author{ Javier Segura\footnotemark[1] \and Soichiro Suzuki\footnotemark[2]
\\
}

\begin{document}

\maketitle

\renewcommand{\thefootnote}{\fnsymbol{footnote}}

\footnotetext[1]{Departamento de Matem\'aticas, Estad\'{\i}stica y
        Computaci\'on. Universidad de  Cantabria, 39005 Santander, Spain. E-mail: javier.segura@unican.es}
        
 \footnotetext[2]{Department of Mathematics, Chuo University, 1-13-27, Kasuga, Bunkyo-ku, Tokyo 112-8551, Japan.
E-mail: soichiro.suzuki.m18020a@gmail.com}

\renewcommand{\thefootnote}{\arabic{footnote}}
\setcounter{footnote}{0}

\begin{abstract}

New sharp bounds for the product and the ratios of modified Bessel functions are presented. Most bounds for the product are derived as direct consequences of previously established bounds for the ratios of consecutive orders, except for the lower bound $I_\nu(x)K_\nu(x) > \frac{1}{2}(x^2 + \nu^2 + 1/5)^{-1/2}$, which had been conjectured for $x > 0$ and $\nu > -1$ and we prove in the present paper, showing that the constant $1/5$ can not be lowered. 
Moreover, very sharp bounds are obtained for the ratios (and consequently for the product) by asymptotically optimizing certain uniparametric inequalities. These optimized bounds are remarkably accurate: they remain extremely sharp for both small and large $x$ with fixed $\nu$, and for large $\nu$ with fixed $x$ or fixed $z = x/\nu$. As a consequence, 
they provide precise upper and lower estimates across a wide range of parameters.

\end{abstract}



\section{Introduction}

The modified Bessel functions ratios $I_{\nu-1}(x)/I_{\nu}(x)$ and $K_{\nu+1}(x)/K_{\nu}(x)$
are important special functions appearing in a vast number of applications. In many instances, simple approximations
for these ratios, and particularly bounds for the ratios, are crucial for extracting relevant information from the expressions where these
ratios appear. The examples of applications where these bounds are used are numerous; see, for example, the list of references given in \cite{Seg11,Seg21}.
 The products $I_{\nu}(x)K_{\nu}(x)$ and related quantities also pop out in a considerable 
 number of applications, see for instance \cite{Elli:2010:CEA,Fucci:2024:VEO,Gigante,Ikoma,Truc:2012:EBF}. Some bounds for the product 
 were described in \cite{Baricz:2016:BFT,Gaunt,Seg21}, however, as we will see, most of the available bounds for the product can be easily improved
  by simple bounds 
derived from known irrational 
bounds for the ratios. 

 In this paper we present three types of results. Firstly, in section \ref{nearop} 
 we obtain new simple bounds for the product of modified Bessel functions which improve or
 complement previous bounds. In the second place, in section \ref{Debye} 
 we investigate the recently conjectured bound (see \cite[Conjecture 1]{Seg21}) 
 $2 I_{\nu}(x)K_{\nu}(x)>(x^2+\nu^2+1/5)^{-1/2}$ ($x>0$, $\nu>-1$) 
  and we prove it by using differential arguments complemented with
 Debye-type asymptotics for large $\nu$. 
 This
 inequality is particularly sharp in the direction $x=\sqrt{2/3}\nu$ as $\nu\rightarrow +\infty$ and it 
 is an interesting touchstone for assessing the sharpness of the existing bounds 
for the ratios of modified Bessel functions and the implied bounds for the product. 
In fact,
 it is not possible to derive this bound as a consequence of
  known bounds for the ratios, and sharper bounds are required to approach its sharpness, particularly in 
 the direction $x=\sqrt{2/3}\nu$.
 In section 
\ref{optimized} (and this is the third type of result) we push one step further the sharpness for
the irrational bounds of the ratios by optimizing the uniparametric sets of bounds described in \cite{Seg23}. 
 These new optimized bounds have the distinct property that they are not only 
 extremely sharp for small and large $x$ with fixed $\nu$ and for large $\nu$ with fixed $x$, but also for large $\nu$ with fixed $z=x/\nu$.

Finally, in section \ref{conjr} we consider the optimized bounds in relation to the previously conjectured inequality for the product 
(which is now proved) and show that they imply
 that $2 I_{\nu}(x)K_{\nu}(x)>(x^2+\nu^2+c)^{-1/2}$ for large enough $\nu$ provided that $c>1/5$, quite close
to case $c=1/5$ but not there.

In the process of approaching the $c=1/5$ bound for the product, the most accurate upper and lower bounds for the ratios and the product available so far, 
particularly for large $\nu$ and fixed $z=x/\nu$, are obtained. These bounds can be used as explicit and easily computable approximations 
for the product and the ratios for a wide range
of the variables (even for complex variables away from oscillatory regions) or as accurate starting values for accelerating higher accuracy computations (for instance via the continued fraction for 
$I_{\nu-1}(x)/I_{\nu}(x)$), as discussed in section \ref{appl}.

\section{Sharp and simple bounds for the product}
\label{nearop}

From the identity \cite[10.28.2]{NIST}
$$
I_{\nu}(x)K_{\nu+1}(x)+I_{\nu+1}(x)K_{\nu}(x)=1/x, 
$$
which holds for any $\nu$, we have, substituting $\nu$ by $\nu-1$ the relation
$$
I_{\nu -1}(x)K_{\nu}(x)+I_{\nu}(x)K_{\nu-1}(x)=1/x, 
$$
from which 
\begin{equation}
\label{rel1}
\Frac{I_{\nu-1}(x)}{I_{\nu}(x)}+\Frac{K_{\nu-1}(x)}{K_{\nu}(x)}=\Frac{1}{xI_{\nu}(x)K_{\nu}(x)}.
\end{equation}
This relation was used in \cite{Seg21} to bound the product of modified Bessel functions by using bounds for the ratios.
However, the fact that we can also write
\begin{equation}
\label{rel2}
\Frac{I_{\nu+1}(x)}{I_{\nu}(x)}+\Frac{K_{\nu+1}(x)}{K_{\nu}(x)}=\Frac{1}{xI_{\nu}(x)K_{\nu}(x)},
\end{equation}
was not sufficiently explored. The main reason for this partial omission was that the trigonometric bounds considered in that reference did not
provide simple bounds when using (\ref{rel2}) instead of (\ref{rel1}). However, for the irrational bounds described 
in \cite{Seg11}, and later collected in a systematic way (and generalized) in \cite{Seg23}, 
these two relations can be used simultaneously to produce simple and sharp bounds.

Taking into account the symmetry observed between the irrational bounds for $I_{\nu-1}(x)/I_\nu (x)$ and 
$K_{\nu+1}(x)/K_\nu (x)$ shown in \cite{Seg23} and summarized in Table 1, we prefer to describe the bounds for the product in terms of these ratios. 
 The relation between the bounds for the ratios and the product is given explicitly in two different lemmas for later use.

\begin{lemma}
\label{bopro1}
If $x\Frac{I_{\nu-1}(x)}{I_{\nu}(x)}>L^{(I)}_{\nu}(x)$ for $\nu\ge \nu^{(I)}$ 
and 
$x\Frac{K_{\nu+1}(x)}{K_{\nu}(x)}<U^{(K)}_\nu (x)$ for $\nu \ge \nu^{(K)}$
then
$
(x^2 L^{(I)}_{\nu+1}(x)^{-1}+U^{(K)}_{\nu}(x))^{-1}<K_{\nu}(x)I_{\nu}(x)<(L^{(I)}_{\nu}(x)+x^2 U^{(K)}_{\nu-1}(x)^{-1})^{-1}
$
where the lower bound holds at least for $\nu\ge\max\{\nu^{(I)}-1,\nu^{(K)}\}$ and the upper bound at least for $\nu\ge \max\{\nu^{(I)},\nu^{(K)}+1\}$
\end{lemma}

\begin{lemma}
\label{bopro2}
If $x\Frac{I_{\nu-1}(x)}{I_{\nu}(x)}<U^{(I)}_\nu (x)$ for $\nu\ge\nu^{(I)}$
and 
$x\Frac{K_{\nu+1}(x)}{K_{\nu}(x)}> L^{(K)}_{\nu}(x)$ for $\nu\ge\nu^{(K)}$
then
$
(U^{(I)}_{\nu}(x)+x^2 L^{(K)}_{\nu-1}(x)^{-1})^{-1}<K_\nu (x) I_{\nu}(x)< (x^2 U^{(I)}_{\nu +1}(x)^{-1}+L^{(K)}_{\nu}(x))^{-1},
$
where the lower bound holds at least for $\nu\ge\max\{\nu^{(I)},\nu^{(K)}+1\}$ and the upper bound at least for $\nu\ge\max\{\nu^{(I)}-1,\nu^{(K)}\}$

\end{lemma}

\begin{table}
\label{tablei}
\begin{tabular}{cccccc}
$(p,q)$ & $\alpha$ & $\beta$ & $\gamma$ & $\{\nu_{min}^{(I)},\nu_{min}^{(K)}\}$ & Type\\
\hline
 $(0,1)$ & $\nu\mp 1$ & $\nu\mp 1$ & $1$ & $\{0,-\infty\}$ & \{L,U\}\\
$(2,0)$ & $\nu$ & $\nu$ & $\sqrt{\nu/(\nu\pm 1)}$ & $\{0,1\}$ & \{L,U\}\\
$(0,2)$ & $\nu \mp \frac12$ & $\nu \mp \frac12$ & $1$ & $\{\frac12,-\frac12\}$ &  \{L,U\}\\
$(2,1)$ & $\nu \mp 1$ & $\nu\pm 1$ & $1$ & $\{0,-\infty\}$ &  \{L,U\}\\
$(0,3)$ & $\nu \mp \frac12$ & $\sqrt{\nu^2-\frac14}$ & $1$ & $\{\frac12,\frac12\}$ &  \{L,U\}\\
$(1,0)$ & $\nu$ & $\nu$ & $\sqrt{\nu/(\nu\mp 1)}$ & $\{1,0\}$ & \{U,L\}\\
$(1,1)$ & $\nu$ & $\nu$ & $1$ & $\{0,-\infty\}$ & \{U,L\} \\
$(1,2)$ & $\nu \mp \frac12$ & $\nu\pm \frac12$ & $1$ & $\{0,1/2\}$ &  \{U,L\}\\
$(3,0)$ & $\nu \mp 2$ & $\nu\pm 2$ & $\sqrt{(\nu\pm 2)/(\nu\pm 1)}$ & $\{0,2\}$ &  \{U,L\}\\
\end{tabular}
\caption{Bounds for $\phi^{(I)}_\nu (x)=xI_{\nu-1}(x)/I_{\nu}(x)$ (upper signs) and  $\phi^{(K)}_\nu (x)=xK_{\nu+1}(x)/K_{\nu}(x)$ (lower signs) 
of the form $\alpha+\sqrt{\beta^2+\gamma^2 x^2}$
classified by their accuracy at $x=0$ ($p$) and $x=+\infty$ ($q$). Bound types are denoted L (lower) or U (upper), with the first letter referring to $\phi^{(I)}$ and the second to $\phi^{(K)}$. The minimum values of 
$\nu$ ensuring validity are given by the pairs 
$\{\nu_{min}^{(I)},\nu_{min}^{(K)}\}$, meaning that for $\phi^{(I)}$ (respectively $\phi^{(K)}$) 
the corresponding inequality is valid for $\nu\ge \nu_{min}^{(I)}$ (respectively $\nu\ge \nu_{min}^{(K)}$).
The inequalities are strict except that the bound for $\phi^{(K)}$ coincides with $\phi^{(K)}$
in the $(1,2)$ and $(0,3)$ cases when $\nu=1/2$.
}
\end{table}

We feed the previous two lemmas with
 the bounds for $xK_{\nu+1}(x)/K_{\nu}(x)$ and $xI_{\nu-1}(x)/I_\nu (x)$ of the form $\alpha +\sqrt{\beta^2+\gamma^2 x^2}$ summarized in Table 1. 
 The most natural selections are those bounds for the product built 
from bounds for the ratios with equal sharpness as
$x\rightarrow 0$ or as $x\rightarrow +\infty$ for both $xK_{\nu+1}(x)/K_{\nu}(x)$ and $xI_{\nu-1}(x)/I_\nu (x)$. 

We recall that the notation $(p,q)$ 
used in \cite{Seg23} for labeling the bounds according to their sharpness means that
the first $p$ (respectively $q$) terms in the series as $x\rightarrow 0$ 
(respectively $x\rightarrow +\infty$) of the bounds
coincide with that of the ratio that is being bounded. The bounds with
sharpness $p+q\le 2$ are remarkably simple bounds. The reason for this is that
 $\alpha=\beta$ in these bounds, which allows for simplifications in the radicals, as we next illustrate. The bounds with maximal sharpness 
 ($p+q=3$) are the most accurate bounds, but not so simple; we also give them for completeness.

Applying Lemma \ref{bopro1} (respectively Lemma \ref{bopro2}) to the bounds labeled with \{L,U\} (respectively \{U,L\}) in the Table 1 we get the following
bounds for the product.

\begin{theorem}
\label{cotprod}
The product $I_{\nu}(x)K_{\nu}(x)$ can be bounded by
$$
F_{-1}^{(p,q)}(\nu,x)<I_{\nu}(x)K_{\nu}(x)<F_{+1}^{(p,q)}(\nu,x),\,x>0
$$
where the lower bounds hold for $\nu\ge \nu_L^{(p,q)}$ and the upper bounds for $\nu\ge \nu_U^{(p,q)}$ for the following sets of functions and ranges
\begin{flushleft}
{\small
\noindent
1. $F_{\lambda}^{(0,1)}(\nu,x)=\Frac{1}{\sqrt{\nu^2+x^2}+\sqrt{(\nu-\lambda)^2+x^2}-\lambda},\,\nu_L^{(0,1)}=-1,\, \nu_U^{(0,1)}=0$.

\noindent
2. $F_{\lambda}^{(1,0)}(\nu,x)=\Frac{1}{2\sqrt{\nu^2+\frac{\nu}{\nu+\lambda}x^2}},\,\nu_L^{(1,0)}=1,\, \nu_U^{(1,0)}=0$.

\noindent
3. $F_{\lambda}^{(0,2)}(\nu,x)=\Frac{1}{2\sqrt{(\nu-\lambda/2)^2+x^2}},\,\nu_L^{(0,2)}=-1/2,\, \nu_U^{(0,2)}=1/2$.

\noindent
4. $F_{\lambda}^{(1,1)}(\nu,x)=\Frac{1}{\sqrt{\nu^2+x^2}+\sqrt{(\nu+\lambda)^2+x^2}-\lambda},\,\nu_L^{(1,1)}=0,\, \nu_U^{(1,1)}=-1$.

\noindent
5. $F_{\lambda}^{(2,0)}(\nu,x)=\frac{1}{\sqrt{(\nu-2\lambda)^2+\frac{\nu-2\lambda}{\nu-\lambda}x^2}+
\sqrt{\nu^2+\frac{\nu}{\nu+\lambda}x^2}+2\lambda},\, \nu_L^{(2,0)}=1,\, \nu_U^{(2,0)}=2$.

\noindent
6. $F_{\lambda}^{(1,2)}(\nu,x)=\frac{1}{\nu+\frac{\lambda}{2}+\sqrt{\left(\nu-\frac{\lambda}{2}\right)^2+x^2}+\frac{x^2}
{\nu+\frac{\lambda}{2}+\sqrt{\left(\nu+\frac{3\lambda}{2}\right)^2+x^2}}},\, \nu_L^{(1,2)}=\frac32,\, \nu_U^{(1,2)}=\frac12.$

\noindent
7. $F_{\lambda}^{(2,1)}(\nu, x)=\frac{1}{\nu-\lambda+\sqrt{\left(\nu+\lambda\right)^2+x^2}+\frac{x^2}
{\nu-\lambda+\sqrt{\left(\nu-2\lambda\right)^2+x^2}}},\, \nu_L^{(2,1)}=-1,\, \nu_U^{(2,1)}=0.$

\noindent
8. $F_{\lambda}^{(3,0)}(\nu, x)=\frac{1}{\nu+2\lambda+\sqrt{\left(\nu-2\lambda\right)^2+\frac{\nu-2\lambda}{\nu-\lambda}x^2}+\frac{x^2}
{\nu+\lambda+\sqrt{\left(\nu+3\lambda\right)^2+\frac{\nu+3\lambda}{\nu+2\lambda}x^2}}},\, \nu_L^{(3,0)}=3,\, \nu_U^{(3,0)}=2.$

\noindent
9. $F_{\lambda}^{(0,3)}(\nu, x)=\frac{1}{\nu-\frac{\lambda}{2}+\sqrt{\nu^2-\frac14+x^2}+\frac{x^2}
{\nu-\frac32\lambda+\sqrt{\left(\nu-\lambda\right)^2-\frac14+x^2}}},\, \nu_L^{(0,3)}=\frac12,\, \nu_U^{(0,3)}=\frac32$}
\end{flushleft}
\end{theorem}

\begin{proof}
The proof is immediate, and we give just one example for clarification. Take for instance the case $(p,q)=(0,2)$. We can apply Lemma \ref{bopro1} with the
entries in the corresponding row of Table 1, that is:
$$
x\Frac{I_{\nu-1}(x)}{I_{\nu}(x)}>L_{\nu}^{(I)}(x)=\nu-1/2+\sqrt{\left(\nu-1/2\right)^2+x^2},\,\nu\ge \frac12 =\nu^{(I)}
$$
and
$$
x\Frac{K_{\nu+1}(x)}{K_{\nu}(x)}<U_{\nu}^{(K)}(x)=\nu+1/2+\sqrt{\left(\nu+1/2\right)^2+x^2},\,\nu\ge -\frac12 =\nu^{(K)}.
$$

We start with the lower bound, which holds for $\nu\ge \max\{\nu^{(I)}-1,\nu^{(K)}\}=1/2$. For this we compute
$$
\begin{array}{ll}
\Frac{x^2}{L_{\nu+1}^{(I)}(x)}+U_{\nu}^{(K)}(x)&\!\!\!=\Frac{x^2}{\nu+1/2+\sqrt{(\nu+1/2)^2+x^2}}+\nu+\frac12+\sqrt{(\nu+1/2)^2+x^2} \\
& \\
&  \!\!\!=2 \sqrt{(\nu+1/2)^2+x^2}.
\end{array}
$$
This gives $I_{\nu}(x)K_{\nu}(x)>F^{(0,2)}_{-1} (\nu,x)$. 

We notice that the fact that $\alpha=\beta$ is important in the simplification. We proceed in the same way for the upper bound.
 
\end{proof}

\begin{remark}
The range of validity of the bounds for the products appearing in the previous theorem is determined by the range for the bounds of the ratios, 
however, it may happen 
that the range is larger for the product than implied from the ratios separately. This is the case of the lower bound $F_{-1}^{(1,1)}(\nu,x)$, 
which turns out to be valid also 
for $\nu\ge -1$, as we have checked numerically; we have not tried to prove this fact analytically. For the rest of bounds there is no modification in the ranges.
\end{remark}

\subsection{Comparison with previous bounds for the product. Sharpness of the bounds}
\label{all}

 First we comment on the comparison with some bounds for the product that have explicitly appeared in
the literature \cite{Baricz:2016:BFT,Gaunt,Seg21}. Following this discussion, we analyze in further detail the sharpness of the bounds in
theorem \ref{cotprod} as well as the sharpness for other bounds of the product that can be obtained from alternative bounds from the ratios, in particular 
those in \cite{Nas:1978:RBF,Ruiz:2016:ANT,Seg23}. 

\subsection{Comparison with previous bounds}
The bounds with $p+q\le 2$ are quite simple and they are enough to improve earlier bounds established in the literature (the bounds with 
$p+q=3$ are even sharper but no so simple). 

To begin with, we consider some of the bounds described in \cite[Theorem 2]{Baricz:2016:BFT}. 
The first two bounds in \cite[Theorem 2]{Baricz:2016:BFT} are, in our notation, of type $(1,0)$, and therefore unsharp as $x\rightarrow +\infty$;
 they are expected to be improved by bounds $(p,q)$ with $p,q\ge 1$, particularly for large $x$. 
 The first bound in that theorem is a non-elementary bound that can be written
 $$
 I_{\nu}(x)K_{\nu}(x)\ge \Frac{\Gamma (\nu)}{2 x^\nu}\left(I_{\nu}(x)-L_{\nu}(x)\right),\,x>0,\,\nu\ge 1/2
 $$ 
 where $L_{\nu}(x)$ is the modified Struve function and the equality takes place for $\nu=1/2$. We correct the range of validity, which was
 declared to be $\nu>0$, when it is $\nu\ge 1/2$ ($\nu>1/2$ for the inequality to be strict)\footnote{This mistake seems
  a consequence of the fact that in the proof the inequality
 $K_{\nu}(x)>2^{\nu-1}\Gamma (\nu)x^{-\nu}e^{-x}$ is said to hold for $\nu>0$, when the correct range is $\nu>1/2$. 
 This errata is also on page 283 of \cite{Gaunt:2014}, though
 the correct range is mentioned in the previous page}. That the bound becomes and identity for $\nu=1/2$ implies that it can be very 
 sharp close to this value, but the lower bound $F_{-1}^{(1,1)}$ is found to be 
 sharper for any $x>0$ provided $\nu$ is slightly large than $1$ and 
 it has the additional advantage of being an elementary bound; the bound $F^{(2,1)}_{-1}$ is even sharper. 
 The second bound in  \cite[Theorem 2]{Baricz:2016:BFT} is elementary and is given 
 by the first two terms in the series in power of $x$ of the previous bounds; this bound is weaker and is surpassed by our bounds.
 The third result in \cite[Theorem 2]{Baricz:2016:BFT} is
of $(2,0)$ type, and $F_{-1}^{(2,0)}$ is superior except in a small region for small $x$, while 
$F_{-1}^{(2,1)}$ is superior in all its range. Finally, we also note that \cite[Theorem 4]{Baricz:2016:BFT} is improved
by the bound $F^{(0,2)}_{-1}$ and that the upper bound of \cite[Corollary 1]{Gaunt} is improved by the upper bound $F_{1}^{(1,0)}$. 
Additionally, Theorem \ref{cotprod} provides upper and lower bounds, which
is also an improvement over the results in \cite{Seg21}, where only the bounds $F_{1}^{(0,2)}$ and $F_{-1}^{(1,1)}$ were described. 

The trigonometric lower bound of \cite[Theorem 8]{Seg21}, however, is not surpassed globally by the bounds in Theorem \ref{cotprod};
the trigonometric bound is sharper for $\nu$ large and $x/\nu$ fixed. In section \ref{optimized} 
 we will obtain bounds which are even sharper in this
direction.

\subsection{Sharpness of the bounds}

All the bounds presented in Theorem \ref{cotprod} are sharp as $\nu\rightarrow +\infty$ and inherit the sharpness with respect to $x$ from the
bounds for the ratios used in their derivation\footnote{Notice, however, that the upper bound $F_{+1}^{(1,1)}$ is not sharp as 
$x\rightarrow 0$ in all its range of validity, but only for $\nu\ge 1$. 
In \cite{Seg21} it was erroneously commented that this bound is not sharp as $x\rightarrow 0$}. Because we have both upper ($U_\nu (x)$) 
and lower bounds ($L_\nu (x)$) for the product, we
can measure the sharpness of the bounds by expanding the relative deviations $$R_\nu (x)=2(U_{\nu}(x)-L_{\nu}(x))/(U_{\nu}(x)+L_{\nu}(x))$$
in the different limits. We summarize the different orders of 
approximation that are obtained in the following theorem.

\begin{theorem}
The relative deviations for the bounds of type $(p,q)$ of Theorem \ref{cotprod} have the following orders of approximation in their corresponding limits:
\begin{enumerate}
\item{}${\cal O}(x^{2p})$ as $x\rightarrow 0$ for $\nu>p$ fixed.
\item{}${\cal O}(x^{-q})$ as $x\rightarrow +\infty$ for $\nu$ fixed.
\item{}${\cal O}(\nu^{-1-2p})$  as $\nu\rightarrow +\infty$ for $x>0$ fixed.
\item{}${\cal O}(\nu^{-1})$ as as $\nu\rightarrow +\infty$ for $z=x/\nu$ fixed, $x,\nu>0$.
\end{enumerate}
\end{theorem}

Of course, it is also possible to build other bounds by using other types of bounds for the ratio. Next we mention three additional types of bounds, to which we
will later add two additional types:
\begin{enumerate}
\item{}The 
$(1,3)$-type bounds of \cite[Theorem 5]{Seg23}; the orders in this case are ${\cal O}(x)$, ${\cal O}(x^{-3})$, ${\cal O}(\nu^{-2})$ 
($x$ fixed) and ${\cal O}(\nu^{-1})$ ($x/\nu$ fixed). 
\item{}The bounds obtained from
the iteration of Riccati bounds described in \cite[Theorems 5 and 9]{Ruiz:2016:ANT}, which provide orders ${\cal O}(x^2)$, ${\cal O}(x^{-2})$, ${\cal O}(\nu^{-4})$ 
($x$ fixed) and ${\cal O}(\nu^{-2})$ ($x/\nu$ fixed). 
\item{}The trigonometric lower bound of \cite{Seg21}, which has orders
of approximation ${\cal O}(x^4)$, ${\cal O}(x^{-2})$, ${\cal O}(\nu^{-6})$ 
($x$ fixed) and ${\cal O}(\nu^{-2})$ ($x/\nu$ fixed); in this case, an upper bound was not obtained but it can be 
built by applying a recurrence step for the
ratios.
\end{enumerate}

The three-term recurrence satisfied by the modified Bessel functions \cite[10.29.1]{NIST} can be used to generate sequences of bounds from 
other known bounds. Indeed, $\phi^{(I)}_\nu (x)
=x I_{\nu-1}(x)/I_{\nu}(x)$ and $\phi^{(K)}_\nu (x)=xK_{\nu+1}(x)/K_\nu (x)$ satisfy
\begin{equation}
\label{CF}
\phi^{(I)}_{\nu}(x)=2\nu+\Frac{x^2}{\phi^{(I)}_{\nu+1}(x)},\,\phi^{(K)}_{\nu}(x)=2\nu+\Frac{x^2}{\phi^{(K)}_{\nu-1}(x)}.
\end{equation}
and we see that using a lower (upper) positive bound in the denominator of the right-hand side gives a an upper (lower bound). 
The sharpness of the bounds obtained increases by two units in each step both as $x\rightarrow 
0$ and $\nu\rightarrow +\infty$ for both functions; however, as explained in \cite{Seg11,Seg21}, the sharpness as $x\rightarrow +\infty$ does not improve
and, for the same reason, neither does the sharpness as $\nu\rightarrow +\infty$ with $z=x/\nu$ fixed. These two limits, and in particular the limit
of $\nu$ large with $z$ fixed are the most demanding accuracy criteria. 	

As a final class of bounds, N\.asell \cite{Nas:1978:RBF} showed how to build sequences of lower and upper bounds for the ratios of the $I$ function; 
a similar analysis 
has not been carried out for the $K$ functions, and so we can not test the accuracy of the products. We focus on
accuracy for the ratios of $I$ functions.
With the notation used in  \cite{Nas:1978:RBF}, the rational lower and upper bounds for 
$I_{\nu+1}(x)/I_{\nu}(x)$ are denoted as $L_{\nu,k,m}(x)$ and
$U_{\nu,k,m}(x)$ respectively, and for those bounds we have
$$
R_{\nu,k,m}(\nu z)=2\Frac{U_{\nu,k,m}(\nu z)-L_{\nu,k,m}(\nu z)}{U_{\nu,k,m}(\nu z)+L_{\nu,k,m}(\nu z)}=f(z)+{\cal O}(\nu^{-1}),
$$
with $f(z)$ a rational function which is positive for all $z>0$.
Therefore the 
bounds are unsharp for large $\nu$ and $z=x/\nu$ fixed, even when they are both sharp at $x=0,+\infty$ if $k,\,m  \ge 1$.  
 This is in contrast with the irrational bounds, which are sharp in this limit.

 Summarizing, there are several types of bounds with different degrees of sharpness 
 for the ratios of modified Bessel functions which also imply bounds for the ratios: rational, irrational bounds (simple or after iteration of a 
 Riccati equation) and trigonometric. Bounds are available which can attain high accuracy as $x$ is small or large with $\nu$ fixed and as 
 $\nu$ is large with $x$ fixed, but there appears to be a lack 
 of methods that are very accurate as $\nu\rightarrow +\infty$ with $z=x/\nu$ fixed. The most accurate methods in this direction are the 
 iterated bounds of \cite{Ruiz:2016:ANT} and the trigonometric bounds from \cite{Seg21} (which are both ${\cal O}(\nu^{-2}$))
 and the less accurate in this direction are the rational bounds. 
 The simple irrational bounds collected in Tables 1 and 
 2 of \cite{Seg21} are all of them ${\cal O}(\nu^{-1})$. We will discuss later how the optimization of the parametric bounds given 
 in Theorems 6, 7, 8 and 9 of \cite{Seg21} provides the most accurate bounds in this direction (${\cal O}(\nu^{-3})$).
 
 It is worth stressing again that the construction of iterated bounds using (\ref{CF}) does not improve the accuracy as $x\rightarrow +\infty$ or 
 as $\nu\rightarrow +\infty$ with
 $z=x/\nu$ fixed.
 
 \begin{remark}
 \label{extra}
 In the sequel we will say that a bound is sharp in a given limit if the absolute value of the exponent in the order of approximation is
 $1$ (for example ${\cal O}(\nu^{-1})$ as $\nu\rightarrow +\infty$), very sharp if that value is two and extremely sharp if it is even
 larger (for instance, ${\cal O}(\nu^{-k})$, $k>2$, as $\nu\rightarrow +\infty$). Observe that we assume integer exponents.
 With this nomenclature, the only very sharp bounds as $\nu\rightarrow +\infty$ with $x/\nu$ fixed are the
 iterated bounds of \cite{Ruiz:2016:ANT} and the trigonometric bounds of \cite{Seg21}. Later we obtain bounds which are extremely sharp in this direction, 
 as well as for $x$ large or small with $\nu$ fixed and for $\nu$ large with $x$ fixed.
 \end{remark}

\section{Proving the conjectured bound for the product}
\label{Debye}

Bounding the product in terms of bounds for the ratios may not be the optimal way to proceed because
we do not compute a direct bound on the product, but obtain this as a consequence of two other bounds. 
It may happen that some valid inequalities for the product escape this type of analysis. 
An interesting example is provided by the following
conjecture, which was put forward in \cite{Seg21}.

\begin{conjecture}
\label{conje}
For all $\nu>-1$ and $x>0$ the following holds
$$
I_{\nu}(x)K_{\nu}(x)>\Frac{1}{2\sqrt{x^2+\nu^2+\frac15}}
$$
\end{conjecture}

In \cite{Seg21}, a similar result was proved but with $c=\frac15$ replaced by $c=\frac13$ and for $\nu>0$, using bounds for 
the ratios in terms of trigonometric functions. From these type of bounds it is not possible to lower the constant to $c=1/5$.
However, as we see next, a proof for $c=1/5$ is possible using differential arguments complemented with asymptotic expansions for large $\nu$.

We start with the asymptotic argument, proving the following result by using Debye-type asymptotics for modified Bessel functions.
\begin{theorem}
\label{teorema}
For any $x>0$ and $c\ge 1/5$ the following inequality holds for large positive $\nu$
$$
I_{\nu}(x)K_{\nu}(x)>\Frac{1}{2\sqrt{x^2+\nu^2+c}}.
$$
The smallest possible value for which the inequality holds is $c=1/5$.
\end{theorem}

\begin{proof}
We take the Debye-type asymptotic expansions  for modified Bessel functions of large order \cite[10.41.3-4]{NIST}
and multiply them, leading to
$$
I_{\nu}(\nu z) K_{\nu}(\nu z)\sim \Frac{1}{2\nu\sqrt{1+z^2}}\displaystyle\sum_{k=0}^{\infty}\Frac{a_{2k}(p)}{\nu^{2k}},
$$
where
$$
p=(1+z^2)^{-1/2},
$$
and
$$
a_k (p)=\displaystyle\sum_{i=0}^{k}(-1)^i U_i(p) U_{k-i}(p), 
$$
the coefficients $U_n(p)$ being polynomials which can be computed with the
recursive formula  \cite[10.41.9]{NIST}. We see that $a_{2k+1}(p)=0$ and the first non zero coefficients are
$$
\begin{array}{ll}
a_0 (p)&=U_0(p)^2=1,\\
a_2 (p)&=2U_2 (p)-U_1(p)^2=\frac58 p^2 (p^2-1)(p^2-\frac15),\\
a_4 (p)&=2U_4(p)-2U_3(p)U_1(p)+U_2(p)^2\\
&=\Frac{1}{128}p^4 (p^2-1)(1155 p^6-1617 p^4+553p^2-27).
\end{array}
$$

Now we set $x=\nu z$ and write the inequality we want to prove as
$$
\Frac{1}{2\nu\sqrt{1+z^2}}\left(1+\Frac{a_2 (p)}{\nu^2}+{\cal O}(\nu^{-4})\right)>\Frac{1}{2\sqrt{\nu^2(1+z^2)+c}}.
$$
Then
\begin{equation}
\label{compa}
1+\Frac{a_2 (p)}{\nu^2}+{\cal O}(\nu^{-4})>\left(1+\Frac{c}{\nu^2 (1+z^2)}\right)^{-1/2}=1-\Frac{c}{2\nu^2 (1+z^2)}+{\cal O}(\nu^{-4})
\end{equation}

For this to hold, neglecting ${\cal O}(\nu^{-4})$, it is enough to consider 
$$
a_2 (p)\ge -\Frac{c}{2(1+z^2)},
$$
and therefore
$$
c\ge -2 p^{-2}a_2 (p)=\frac54 (1-p^2)(p^2-\frac15)=f(p).
$$
Then, the minimal constant value $c$ for which this holds for any $\nu$ (sufficiently large) and any $0<p<1$ is
$$
c=\max_{p\in[0,1]}f(p)
$$
For maximizing $f(p)$ we calculate the derivative
$$
f'(p)=p(3-5p^2)=0,
$$
and the required solution is $p=\sqrt{3/5}$ (and then $z=\sqrt{2/3}$). For this value of $p$ we have $c=1/5$.

With this selection, we have that as $\nu\rightarrow +\infty$ the inequality holds strictly if $z\neq \sqrt{2/3}$ 
because the ${\cal O}(\nu^{-2})$ in the left hand side of (\ref{compa}) is larger than that on the right. For $z=\sqrt{2/3}$ both
${\cal O}(\nu^{-2})$ terms are the same and we need to check the next order. For this we use that 
$$
a_2 (\sqrt{3/5})=-3/50,\, a_4(\sqrt{3/5})=783/25000,
$$
and then the analogous to (\ref{compa}) but up to ${\cal O}(\nu^{-4})$ and for the particular case $p=\sqrt{3/5}$ is
\begin{equation}
\label{compa2}
1-\frac{3}{50\nu^2}+\Frac{783}{25000\nu^4}+{\cal O}(\nu^{-6})>1-\frac{3}{50\nu^2}+\frac38\left(\Frac{3}{50}\right)^2\Frac{1}{\nu^4}
+{\cal O}(\nu^{-6}).
\end{equation}
and the inequality is also strict for $z=\sqrt{2/3}$. Therefore, we have proved that
$$
2 I_{\nu}(\nu z)K_{\nu}(\nu z)>\Frac{1}{\sqrt{\nu^2 (1+z^2)+\frac15}} \mbox{ as }\nu\rightarrow +\infty
$$
with $1/5$ the best possible constant, and the theorem is proved.
\end{proof}

We observe that the inequality is particularly sharp if $z=\sqrt{2/3}$ and in this case we have
$$
2 I_{\nu}(\sqrt{2/3} \nu)K_{\nu}(\sqrt{2/3}\nu)-\Frac{1}{\sqrt{\frac52\nu^2+\frac15}}={\cal O}(\nu^{-5})
$$

For general $0<p<1$ we have

\begin{equation}
\label{resta}
2 I_\nu (\nu z)K_\nu (\nu z)-\Frac{1}{\sqrt{x^2+\nu^2+1/5}}
= \Frac{p}{\nu^3}\left(\alpha_0 (p)+\Frac{\alpha_2(p)}{\nu^2}+{\cal O}(\nu^{-4})\right)
\end{equation}
where
$$
\alpha_0 (p)=\Frac{1}{40}p^2 (5 p^2-3)^2,
$$
which shows again that the inequality holds for any $\nu$ sufficiently large and any $0<p<1$ neglecting order ${\cal O}(\nu^{-2})$ in (\ref{resta}).
 
 We should stress the peculiarity of the bound under discussion in the sense that, contrarily to all the bounds studied so far for modified Bessel functions,
 its greatest sharpness is not in any of the usual limits $x\rightarrow 0$, $x\rightarrow +\infty$ with $\nu$ fixed 
 or $\nu\rightarrow +\infty$ with $x$ fixed, but in the limit $\nu\rightarrow +\infty$ with $x/\nu=\sqrt{2/3}$ and 
 the bound becomes extremely sharp along this line.
 
 One may consider the possibility of proving Conjecture \ref{conje} using any of the irrational, rational or trigonometric bounds described in 
section \ref{nearop}. However, this analysis is doomed to fail because, as 
we saw earlier (see for example Eq. (\ref{compa2})) the first two terms in the expansion as $\nu\rightarrow +\infty$ with $x/\nu=\sqrt{2/3}$ 
of the conjectured bound coincide with those of the asymptotic expansion for $I_{\nu}(\sqrt{2/3}\nu)K_{\nu}(\sqrt{2/3}\nu)$. This means
that a bound that improves the conjecture should have an order of approximation at least ${\cal O}(\nu^{-3})$ with $x/\nu$ fixed, and
the best accuracy available so far is ${\cal O}(\nu^{-2})$ for the iterated bounds of \cite{Ruiz:2016:ANT} and the trigonometric bounds of
\cite{Seg21}. With approximations of order ${\cal O}(\nu^{-2})$ we can only prove weaker versions of the conjecture: from the trigonometric bounds
it was proved in \cite{Seg21} that $I_\nu (x) K_\nu (x)>1/(2\sqrt{x^2+\nu^2+c})$ 
with $c=1/3$, and with the bounds in \cite{Ruiz:2016:ANT} 
a similar but even weaker inequality could be proved.\footnote{Using those bounds, it can be proved that the lower bound for the product is larger 
than $1/(2\sqrt{x^2+\nu^2+c})$ for all positive $x$ and $\nu$ large if $c\ge (3/4)^3$.}

However, surprisingly, the conjecture can be proved with quite elementary differential arguments, as we next show. 

In the sequel, we will use the following definitions (we drop the order $\nu$ in the notation)
\begin{equation}
\label{defi}
p (x)=I_\nu (x)K_\nu (x),\,q (x)=\Frac{1}{4p(x)^2},\,h(x)=q(x)-x^2-\nu^2.
\end{equation}
Proving this conjecture is therefore equivalent to showing that $h(x)<1/5$ for all $\nu>-1$ and $x>0$.
For this purpose, we give two auxiliary lemmas, then we prove the result for the non-strict inequality ($h(x)\le 1/5$) and finally show that the 
inequality is strict. 

\begin{lemma}
\label{masterl}
The function $q$ of Eq. (\ref{defi}) satisfies the following differential equation
\begin{equation}
\label{master}
-4q(x)(x^2q''(x)+xq'(x))+5(xq'(x))^2 +16q(x)^2 h(x)=0
\end{equation}
\end{lemma}
\begin{proof}
Denoting $u=I_\nu$ and $v=K_\nu$ and considering that both functions are solutions of the ODE
$$
x^2 y''+xy'-(x^2+\nu^2)y=0
$$
we have
\begin{equation}
\label{deri}
p''=u''v+2u'v'+u v''=2\left(1+\Frac{\nu^2}{x^2}\right)p-\Frac{1}{x}p'+2u'v'.
\end{equation}
On the other hand we have the Wronskian relation
$
u'v-uv'=1/x
$
and then 
$$
p'^2-\Frac{1}{x^2}=(u'v+uv')^2-(u'v-uv')^2=4uvu'v'=4pu'v'
$$
and using this to eliminate $u'v'$ in Eq. (\ref{deri}) we get
\begin{equation}
\label{casi}
p'^2-\Frac{1}{x^2}=2pp''+\Frac{2}{x}pp'
-4\left(1+\Frac{\nu^2}{x^2}\right)p^2.
\end{equation}
Now, because 
$p=1/2q^{-1/2}$, differentiation of $p$ and substitution in (\ref{casi}) completes the proof.
\end{proof}

\begin{remark}
Eq. (\ref{master}) can also be written as
\begin{equation}
\label{enlay}
-4q \ddot{q}+5\dot{q}^2+16q^2 h=0
\end{equation}
where dots mean derivative with respect to $y$, $y=\log (x)$. The proofs can also be carried in the variable $y$, and 
some computations can be shortened, but we give the results in terms of the original variable. 
\end{remark}

\begin{lemma}
\label{bound}
If $\nu>-1$ then $h(0^+)\le 0$ and $h(+\infty)<0$.
\end{lemma}
\begin{proof}
The proof follows by a straightforward use of the limiting properties of modified Bessel functions. We start with the limit $x\rightarrow 0^+$
considering three cases:
\begin{enumerate}
\item{If} $\nu>0$ then, as $x\rightarrow 0^+$, $I_{\nu}(x)\sim (x/2)^\nu/\Gamma (\nu+1)$ and $K_{\nu}(x)\sim (x/2)^{-\nu}\Gamma (\nu)/2$ (see 
\cite[\S 10.30(i)]{NIST}) and then $I_{\nu}(x)K_{\nu}(x)\sim 1/(2\nu)$, which gives $h(0^+)=0$.
\item{If} $\nu=0$, as $x\rightarrow 0$, $I_{0}(x)\sim 1$ and $K_0 (x)\sim -\log x$, and $h(0^+)=0$.
\item{If} $\nu\in (-1,0)$, $K_{\nu}(x)=K_{-\nu}(x)$ and because $I_{\nu}(x)\sim (x/2)^{\nu}/\Gamma (\nu+1)$ if $\nu>-1$, we have 
 $I_{\nu}(x)K_{\nu}(x)\rightarrow +\infty$
as $x\rightarrow 0^+$ and $h(0^+)=-\nu^2<0$.
\end{enumerate}
With respect to the limit as $x\rightarrow +\infty$, considering the Hankel asymptotic expansions \cite[\S 10.40(i)]{NIST} we have
$$
I_{\nu}(x)K_{\nu}(x)=\Frac{1}{2x}-\Frac{\nu^2-1/4}{4x^3}+{\cal O}(x^{-4}),
$$
which gives 
$$
h(x)=-\Frac{1}{4}+{\cal O}(x^{-2}).
$$
Therefore $h(+\infty)=-1/4<0$.
\end{proof}

Next, we prove the non-strict inequality\footnote{The first proof for this inequality (due to S. Suzuki) used a
differential equation similar to that in Lemma \ref{masterl}, but in terms of $h(x)$ and 
$P(x)=(2I_\nu (x) K_\nu(x))^{-1}$ and with the change of variable $y=2\log (x)$. The main idea shared by the proof given in this paper and
this earlier proof is to bound $h(x)$.}

\begin{lemma}
\label{noes}
If $\nu>-1$ then $h(x)\le 1/5$ for all $x>0$
\end{lemma}

\begin{proof}
$h$ is differentiable in ${\mathbb R}^+$ and because $h(0^+)\le 0$ and $h(+\infty)<0$ it is non-positive, and the 
result holds trivially, or it may reach its maximum value in ${\mathbb R}^+$ at some local extrema $x_0$ where 
$h(x_0)>0$, $h'(x_0)=0$ and $h''(x_0)\le 0$ (and then, as a consequence, $q'(x_0)=2x_0$ and $q''(x_0)=2+h''(x_0)\le 2$). 
Substituting in  Eq. (\ref{master}):
\begin{equation}
\label{mas2}
5x_0^4-x_0^2 	q(x_0)q''(x_0)-2x_0^2 q(x_0)+4q(x_0)^2 h(x_0)=0
\end{equation}

Now, because $q''(x_0)\le 2$ and $q(x_0)>0$, we have that
$$
5x_0^4-4x_0^2 	q(x_0)+4q(x_0)^2 h(x_0)\le 0
$$
Since we are assuming that $h(x_0)>0$ this is, as a function of $q(x_0)$, a parabola that opens upward, and so the
previous inequality is only possible if the discriminant is non-negative. Because the discriminant is
$$
16 x_0^4 (1-5h(x_0))\ge 0,
$$
this forces $h(x_0)\le 1/5$. Therefore $\sup_{x>0}h(x)\le 1/5$.
\end{proof}

Next we prove that the inequality is strict, which closes the proof of the main result in this section.

\begin{theorem}
\label{theorema}
Let $\nu>-1$ then $h(x)<1/5$ for all $x>0$, in other words
$$
I_\nu (x) K_\nu (x)>\Frac{1}{2\sqrt{x^2+\nu^2+1/5}},\,x>0.
$$
The value $1/5$ is the smallest possible constant for which the inequality holds.
\end{theorem}
\begin{proof}
That $1/5$ is the best possible constant was already proved, and we also know that the non-strict inequality holds. Now, we have to prove that
the equality is never reached, in other words, that no $x_0$ exists such that $h(x_0)=1/5$. Solving for $q''(x_0)$ in 
Eq. (\ref{mas2}) we have
$$
q''(x_0)=\Frac{5x_0^2}{q(x_0)}+\Frac{4q(x_0)h(x_0)}{x_0^2}-2
$$
Assume now that $h(x_0)=1/5$ and denote $t=q(x_0)/x_0^2$, then
$$
q''(x_0)-2=\Frac{5}{t}+\frac45 t -4=\Frac{(2t-5)^2}{5t}
$$
We conclude that $q''(x_0)\ge 2$, but because $h''(x_0)\le 0$ implies that $q''(x_0)\le 2$ we see that $h(x_0)=1/5$ implies
$q''(x_0)=2$. Therefore, the numerator must be zero, that is $t=5/2=q(x_0)/x_0^2=(h(x_0)+x_0^2+\nu^2)/x_0^2=(1/5+x_0^2+\nu^2)/x_0^2$ which gives
$$
\nu^2=\frac32 x_0^2-\frac15.
$$
Therefore if $h(x_0)=1/5$ we have $q(x_0)=\frac52 x_0^2$, $q'(x_0)=2x_0$ and $q''(x_0)=2$, and then $h'(x_0)=h''(x_0)=0$. Now,
taking the derivative of (\ref{master}) with respect to $x$ and setting $x=x_0$, after some 
algebra\footnote{For the computation of the derivatives of $q$, we find more economic to use the equation in the $y$ variable 
(\ref{enlay}) 
to differentiate with respect to this variable, and then restore the result in terms of the original variable} 
we get $q'''(x_0)=0$, and therefore $h'''(x_0)=0$.

Taking now the second derivative of (\ref{master}) and setting $x=x_0$, after some elementary computations we get
$$
q^{(4)}(x_0)=\Frac{144}{25x_0^2},
$$
and then
$$
h^{(4)}(x_0)=q^{(4)}(x_0)>0.
$$
Because $h(x_0)=1/5$, $h'(x_0)=h''(x_0)=h^{(3)}(x_0)=0$ and $h^{(4)}(x_0)>0$, Taylor's theorem gives that, around $x_0$
$$
h(x)=\frac15+\frac{1}{4!}h^{(4)}(x_0)(x-x_0)^4+o ((x-x_0)^4),
$$
and since $h^{(4)}(x_0)>0$ we have $h(x)>1/5$ for $x$ close to $x_0$, which is not possible because we already proved in the 
previous lemma that 
$h(x)\le 1/5$.
\end{proof}

\section{Optimized irrational bounds}
\label{optimized}

As discussed earlier, it is not possible to prove the inequality of Theorem \ref{theorema} as a consequence of the 
existing bounds for the ratios of modified Bessel functions. We would need sharper bounds, particularly in the direction
$x=\sqrt{2/3}\nu$ for large $\nu$. 
The uniparametric families of bounds of Theorems 6, 7, 8 and 9 of \cite{Seg23} 
seem particularly fit for this purpose: 
by choosing an admissible value of the parameter the direction with $z=x/\nu$ fixed and large $\nu$ for which the bound is the sharpest can be
 fixed.

We condense these four theorems of \cite{Seg23} in the Theorems \ref{para1} and \ref{para2} of the present paper. 
We first discuss Theorem \ref{para1}, which gives parametric bounds which are sharp at $+\infty$ for any value
of the parameter, and analyze the optimized bounds that can be constructed from this result (in Theorem \ref{para2}, the
bounds are sharp at $x=0$ for any admissible value of the parameter).

\begin{theorem}
\label{para1}
Let $\lambda\in[0,1/2]$ and $\nu\ge 1/2-\lambda$. Denoting
$$
\alpha_{\nu}^{\pm}(\lambda)=\nu\mp(1/2+\lambda),\qquad
\beta_{\nu}^{\pm}(\lambda)=\pm\sqrt{\,2\lambda}+\sqrt{\nu^2-(\lambda-1/2)^2\,},
$$
and
$$
B(\alpha,\beta,\gamma,x)=\alpha+\sqrt{\beta^2+\gamma^2 x^2},
$$
the following holds for all $x>0$
\[
x\frac{I_{\nu-1}(x)}{I_\nu(x)} > B(\alpha_\nu^{+}(\lambda),\beta_\nu^{+}(\lambda),1,x),\,
x\frac{K_{\nu+1}(x)}{K_\nu(x)} < B(\alpha_\nu^{-}(\lambda),\beta_\nu^{-}(\lambda),1,x)
\]
\end{theorem}

When $\lambda=0$ the bounds in this theorem are the $(0,3)$ bounds of Table 1, and for $\lambda=1/2$ they are the $(2,1)$ bounds.

In order to find approximately the value of $\lambda$ which gives the highest accuracy for a given $z=x/\nu$ 
we minimize (respectively maximize) the upper (respectively) lower bounds with respect to $\lambda$ in the limit $\nu\rightarrow +\infty$ with
 $z=x/\nu$ fixed. The critical
 value of $\lambda$ is found by solving $\Frac{d}{d\lambda}B(\alpha_\nu^+ (\lambda),\beta_\nu^+ (\lambda),1,x)=0$, with 
 $\alpha_\nu^+(\lambda)$ and $\beta_\nu^+(\lambda)$ as given in Theorem \ref{para1}. This leads to
 $$
 x^2=(\beta_\nu^{+})^2 \left(\left(\Frac{d\beta^+_\nu /d\lambda}{d\alpha^+_\nu /d\lambda}\right)^2 -1\right),
 $$
 which we solve for $\lambda$ in the limit $\nu\rightarrow +\infty$ with $z=x/\nu$. Expanding the previous equation we have
 $$
 0=\Frac{1}{\nu^2}(\beta_\nu^+)^2 \left(\left(\Frac{d\beta^+_\nu /d\lambda}{d\alpha^+_\nu /d\lambda}\right)^2 -1\right)-z^2=\Frac{1}{2\lambda}-(1+z^2)+{\cal O}(\nu^{-1}),
 $$
 and, neglecting the ${\cal O}(\nu^{-1})$ term, the solution is
 $$
 \lambda=\Frac{1}{2(1+z^2)}.
 $$
 We observe that $\lambda =(2(1+z^2))^{-1}\in [0,1/2]$, as required by \cite[Theorem 6]{Seg23}; the theorem also requires 
 $\nu>1/2-\lambda$. For the case of the parametric bound for the ratio of modified Bessel functions of the second kind 
 we obtain the same optimal value of $\lambda$. 
 
 Given a fixed value of $z=x/\nu$, by selecting the value of $\lambda=(2(1+z^2))^{-1}$ we obtain a bound that is particularly sharp in 
 the direction of fixed $z=x/\nu$ and large $\nu$. For instance, taking $\lambda=3/10$ in
  Theorem  \ref{para1}, the resulting bound is specially sharp in the direction $z=x/\nu=\sqrt{2/3}$. 
  A better idea to improve the sharpness for any $z$ 
  is to insert $\lambda=(2(1+z^2))^{-1}=\nu^2/(2(x^2+\nu^2))$ as a function of $\nu$ and $z$. 
 With this substitution in Theorem \ref{para1} we obtain the following optimized bounds.
 
 \begin{theorem}
 \label{optiinf}
 Let $x>0$ and $\nu\ge 1/2$ then the following holds
 $$
 \begin{array}{ll}
 x\Frac{I_{\nu-1}(x)}{I_\nu (x)}>L^{(I)}_\nu(x)=& \nu-\frac12\left(1+\Frac{\nu^2}{\nu^2+x^2}\right)\\
 &+\sqrt{x^2+\left(\Frac{\nu}{\sqrt{\nu^2+x^2}}+\sqrt{\nu^2-\Frac{x^4}{4(x^2+\nu^2)^2}}\right)^2}
 \end{array}
 $$
 
 $$
 \begin{array}{ll}
 x\Frac{K_{\nu+1}(x)}{K_\nu (x)}<U^{(K)}_\nu(x)=& \nu+\frac12\left(1+\Frac{\nu^2}{\nu^2+x^2}\right)\\
 &+\sqrt{x^2+\left(-\Frac{\nu}{\sqrt{\nu^2+x^2}}+\sqrt{\nu^2-\Frac{x^4}{4(x^2+\nu^2)^2}}\right)^2}
 \end{array}
 $$
 \end{theorem}
 
 The range of $\nu$ in the previous theorem can be slightly extended to $\nu\ge 1/2-\lambda=x^2/(2(\nu^2+x^2))$.

 From these bounds one can build lower and upper bounds for the product considering Lemma \ref{bopro1}. We don't write these bounds explicitly. 
 The expressions for these bounds are quite involved, but they are very accurate, as we will later illustrate, and both can be used as relatively simple
 approximations for the product at least for $\nu>3/2$. 
 \begin{remark}
 \label{notation1}
 The parametric bounds (\ref{para1}) used in the construction of optimized bounds (\ref{optiinf}) 
 are sharp as $x\rightarrow +\infty$ for any parameter selection
 (also for the optimal selection), and so are the related bounds for the product. For this reason, we 
 denote the lower and upper bounds for the product that are obtained using Theorem \ref{optiinf} and Lemma \ref{bopro1} 
 as $L_{\nu}^{(+\infty)}(x)$ and
 $U_{\nu}^{(+\infty)}(x)$ respectively. For all $x>0$, $L_{\nu}^{(+\infty)}(x)$ holds for $\nu\ge 1/2$ and
 $U_{\nu}^{(+\infty)}(x)$ for $\nu\ge 3/2$.
 \end{remark}
 
 Now we proceed similarly with the parametric bounds of Theorems 8 and 9 of \cite{Seg23}, that we condense in the following theorem:
\begin{theorem}
\label{para2}
Let $\lambda\in[1/2,2]$ and
$$
 B_{\nu}^{\pm}(\lambda,x)=\nu \mp \lambda +\sqrt{(\nu \pm \lambda)^2+ h(\lambda)x^2},
$$
$$
h(\lambda)=\Frac{\nu\pm \lambda}{\nu\mp f(\lambda)},\,f(\lambda)=\lambda-2\sqrt{2\lambda}+1
$$
Then, for all $x>0$,
$$
x\,\frac{I_{\nu-1}(x)}{I_\nu(x)} < B_\nu^+ (\lambda), \,x\,\frac{K_{\nu+1}(x)}{K_\nu(x)} >B_\nu^{-}(\lambda).
$$
\end{theorem}

As done previously for Theorem \ref{para1}, we optimize the bounds of Theorem \ref{para2} by setting $d B_{\nu}(\lambda,x)/d\lambda=0$,
which leads to
$$
x^2=4\Frac{h(\lambda)-(\nu+\lambda)h'(\lambda)}{h'(\lambda)^2}.
$$
As before, we set $x=\nu z$ and we solve the equation in the asymptotic limit by expanding for large $\nu$ and fixed $z$:
$$
0=\Frac{4}{\nu^2}\Frac{h(\lambda)-(\nu+\lambda)h'(\lambda)}{h'(\lambda)^2}-z^2=2\Frac{\sqrt{2\lambda}-\lambda}{(1-\sqrt{2\lambda})^2}-z^2+
{\cal O}(\nu^{-1}).
$$
Neglecting ${\cal O}(\nu^{-1})$ and solving for $\lambda$ we have
$$
\lambda=\frac12\left(1+\Frac{1}{\sqrt{1+z^2}}\right)^2=\frac12 \left(1+\Frac{\nu}{\sqrt{\nu^2+x^2}}\right)^2,
$$
and we observe that $\lambda \in [1/2,2]$ as corresponds to the range of values in Theorem \ref{para1}.

Substituting now this value of $\lambda$ in the parametric bounds of Theorem \ref{para2} we get the following
optimized bounds. 

\begin{theorem}
\label{optiinf2}
Let $p=\nu/\sqrt{\nu^2+x^2}$ and $c_{\pm}(\nu,p)=\Frac{2\nu\pm (1+2p+p^2)}{2\nu \pm (1+2p-p^2)}$ then we have that for all $\nu\ge 0$
$$
x\Frac{I_{\nu-1}(x)}{I_\nu (x)}<U_{\nu}^{(I)}(x)=\nu-\Frac{(1+p)^2}{2}+\sqrt{\left(\nu+\Frac{(1+p)^2}{2}\right)^2+c_+ (\nu,p)x^2},
$$
and for $\nu\ge 2$
$$
x\Frac{K_{\nu+1}(x)}{K_\nu (x)}>L_{\nu}^{(K)}(x)=\nu+\Frac{(1+p)^2}{2}+\sqrt{\left(\nu-\Frac{(1+p)^2}{2}\right)^2+c_- (\nu,p)x^2},
$$
\end{theorem}
The range of validity for the second bound in the previous theorem can be slightly extended to $\nu>\lambda=(1+p)^2 /2$.

From these bounds one can build lower and upper bounds for the product considering Lemma \ref{bopro2}, which provide relatively simple
 approximations for the product at least for $\nu>3$. 
 
 \begin{remark}
 \label{notation2}
 The parametric bounds (\ref{para1}) used for building the optimized bounds (\ref{optiinf2}) are sharp at $x=0$ for any parameter selection
 (also for the optimal selection); for this reason we denote the lower and upper bounds for the product which are
  obtained using Theorem \ref{optiinf2} and Lemma \ref{bopro2} as $L_{\nu}^{(0)}(x)$ and
 $U_{\nu}^{(0)}(x)$ respectively. For all $x>0$, $L_{\nu}^{(0)}(x)$ holds for $\nu\ge 3$ and
 $U_{\nu}^{(0)}(x)$ for $\nu\ge 2$.
 \end{remark}

 Next, we analyze the sharpness for the bounds of the product which are obtained from the optimized bounds for the ratios (for the 
 ratios a similar analysis is of course possible). As we did previously, we consider the relative deviations
 \begin{equation}
 \label{deviations}
 \begin{array}{l}
 R^{(0)}_\nu (x)=2\Frac{U_{\nu}^{(0)}(x)-L_{\nu}^{(0)}(x)}{U_{\nu}^{(0)}(x)+L_{\nu}^{(0)}(x)},\,
 R^{(+\infty)}_\nu (x)=2\Frac{U_{\nu}^{(+\infty)}(x)-L_{\nu}^{(+\infty)}(x)}{U_{\nu}^{(+\infty)}(x)+L_{\nu}^{(+\infty)}(x)}
 \end{array}
 \end{equation}
 and take the expansions in the different limits. The orders of approximation in the different limits are shown in Table 2.
 We observe that both bounds have orders $R={\cal O}(\nu^{-3})$ in the limit $\nu\rightarrow +\infty$ with $x/\nu$ fixed and therefore 
 are extremely sharp in this limit (see Remark \ref{extra}), 
 differently to previous bounds. Furthermore, the $U^{(+\infty)}$ and $L^{(+\infty)}$ are extremely sharp in all directions, and the same
 is true for the $U^{(0)}$ and $L^{(0)}$ except at $x=+\infty$ (in this limit they are only very sharp).

 \begin{table}[h]
\label{sisisi}
\begin{center}
\begin{tabular}{ccccc}
 & $x\rightarrow 0$ & $x\rightarrow +\infty$ & $\nu\rightarrow +\infty$ & $\nu\rightarrow +\infty$\\
 &   $\nu$ fixed    &      $\nu$ fixed      &  $x$ fixed  & $z=x/\nu$ fixed \\
 &&&&\\
$R^{(+\infty)}$ & ${\cal O}(x^{4})$ & ${\cal O}(x^{-3})$ & ${\cal O}(\nu^{-7})$& ${\cal O}(\nu^{-3})$\\
&&&&\\
$R^{(0)}$ & ${\cal O}(x^{6})$ & ${\cal O}(x^{-2})$ & ${\cal O}(\nu^{-9})$& ${\cal O}(\nu^{-3})$
\end{tabular}
\end{center}
\caption{Sharpness of the bounds for the products obtained from optimized bounds for the ratios
 }
\end{table}

For small $\nu$ all the bounds we have obtained tend to worsen, but even for quite small $\nu$ 
the bounds are quite accurate. For instance, we have checked that for $\nu\ge 4$
the four upper and lower bounds approximate the product $I_{\nu}(x) K_{\nu}(x)$ with a relative accuracy smaller than $10^{-3}$ for any $x>0$; of course as $\nu$ becomes larger,
the accuracy improves. 
Therefore, relatively simple approximations can be used for bounding accurately the product and the ratios for any values of $x$ and $\nu$, but
$\nu$ not too small. The same is true for the ratios.

The accuracy of the bounds for the product is illustrated Figure 1, where the level curves $R_{\nu}^{(+\infty)}(x)=\epsilon$ and 
$R_{\nu}^{(0)}(x)=\epsilon$ are plotted for $\epsilon=10^{-4},\,10^{-5}\,,10^{-6}$.

\begin{figure}
\begin{center}
\begin{minipage}{8cm}
 \includegraphics[width=1\textwidth]{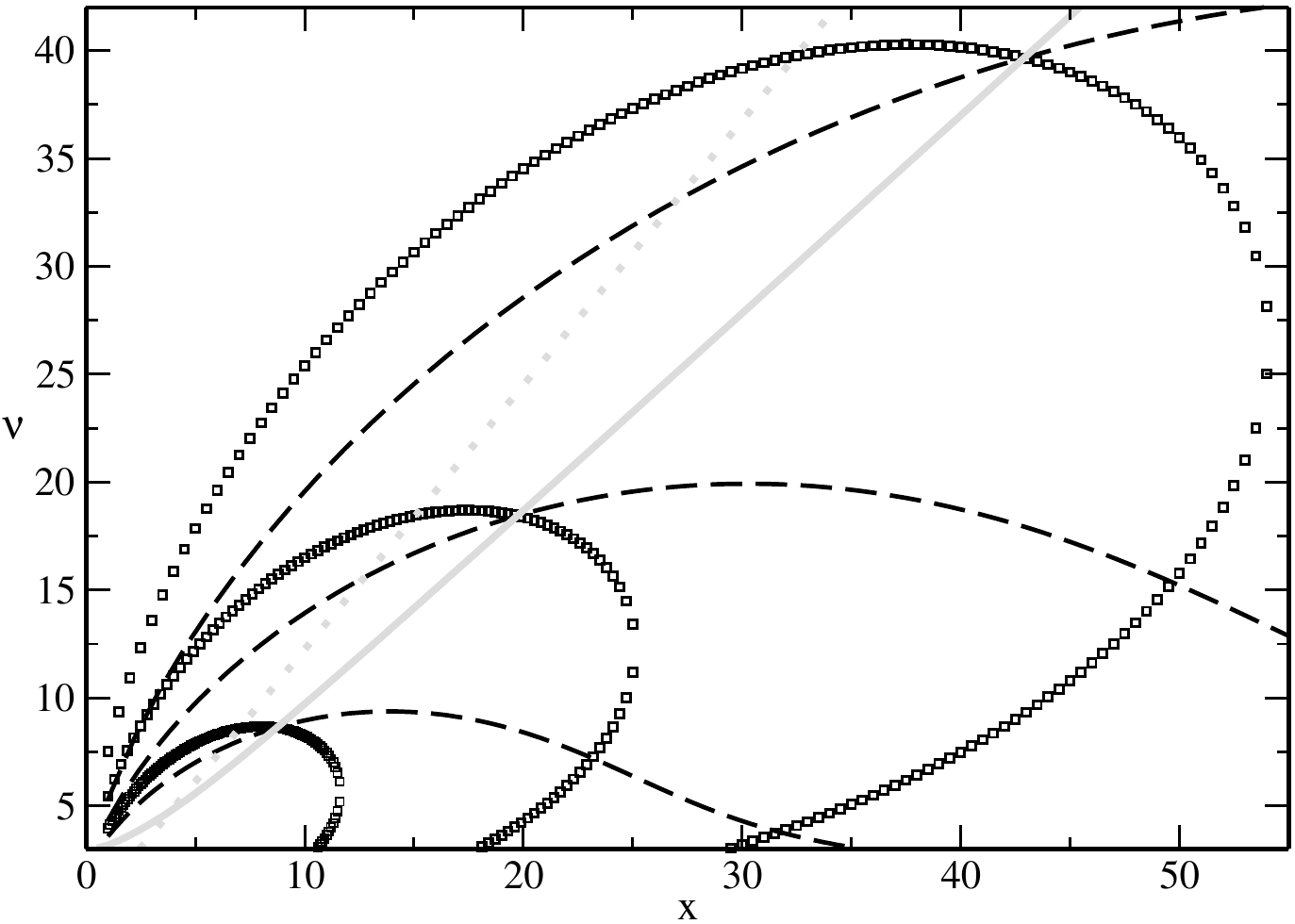}
\end{minipage} 
\end{center}
 \caption{Level curves for the accuracy of the bounds.
 The daisy petal shaped curves correspond to $R_{\nu}^{(+\infty)}(x)=\epsilon$, while the dashed lines
 represent $R_{\nu}^{(0)}(x)=\epsilon$ (see Eq. (\ref{deviations})). 
 The level curves shown are for $\epsilon=10^{-4},\,10^{-5}\,,10^{-6}$. The gray solid line represents the curve 
 $R_{\nu}^{(0)}(x)=R_{\nu}^{(+\infty)}(x)$, with $R_{\nu}^{(0)}(x)<R_{\nu}^{(+\infty)}(x)$ above this curve. The grey dotted line corresponds to
 $z=\sqrt{2/3}$}
\end{figure}

\section{Approaching the sharp inequality for the product from optimized irrational bounds}
\label{conjr}

The previous optimized bounds are not enough to improve the sharpness of the inequality for the product with constant $c=1/5$ in the critical direction
$x=\sqrt{2/3}\nu$, but they are remarkably close.
To see this, we consider the optimized lower bounds previously discussed and we compare them with the bound 
$L_\nu^{(c)} (x)=
\frac12(x^2+\nu^2+c)^{-1/2}$ in the limit $\nu\rightarrow +\infty$ with $z=x/\nu$ fixed. Expanding and writing the result as a 
function of $p=1/\sqrt{1+z^2}$. We have
\begin{equation}
\label{B0}
L^{(0)}_\nu(\nu z) -L_\nu^{(c)}(\nu z)=\Frac{p^3}{4\nu^3}\left(c-f(p)-\frac{g(p)}{\nu}+{\cal O}(\nu^{-2})\right)
\end{equation}
where
$$
f(p)=\frac14 (1-p^2)(5p^2-1),\,g(p)=\frac14 (1-p)^{3}\,(1+p)\,\bigl(1+6p+10p^{2}\bigr)
$$

A similar thing happens with the difference $L^{(+\infty)}_\nu(\nu z) -L_\nu^{(c)}(\nu z)$; the first term in the expansion is as before. This
was expected because, as explained earlier, the relative deviations for the optimized bounds are of order ${\cal O}(\nu^{-3})$ and the 
bounds are ${\cal O}(\nu^{-1})$. We have
\begin{equation}
\label{Binf}
L^{(+\infty)}_\nu(\nu z) -L_\nu^{(c)}(\nu z)=\Frac{p^3}{4\nu^3}\left(c-f(p)-\Frac{h(p)}{\nu}+{\cal O}(\nu^{-2})\right)
\end{equation}
where $f(p)$ is as before and
$$
h(p)=\frac52 p^3\,(1-p)^{2}\,(1+p)
$$

Because the absolute maximum of $f(p)$ for $p\in [0,1]$ is reached at $p=\sqrt{3/5}$ ($z=\sqrt{2/3}$), 
where $f(\sqrt{3/5})=1/5$ it is clear that both $L^{(0)}_\nu(\nu z)$ 
and $L^{(+\infty)}_\nu(\nu z)$
are better bounds that $L_\nu^{(c)}(x)$ for any $x$ provided $c>1/5$ and $\nu$ is large enough, being $c-f(p)>0$ (see Eqs. (\ref{B0}) and 
(\ref{Binf})). 
However, the same is not true for $c=1/5$, because the next term in Eqs. (\ref{B0}) and  (\ref{Binf}) is negative.

 \section{Some applications}
 \label{appl}
 
 As mentioned earlier, the examples of applications where the bounds for ratios of modified Bessel functions are useful are numerous (see, for example, 
 \cite{Seg11,Seg21} and references cited therein) and the products $I_{\nu}(x)K_{\nu}(x)$ and related quantities also appear in a considerable 
 number of applications, see for instance \cite{Elli:2010:CEA,Fucci:2024:VEO,Gigante,Ikoma,Truc:2012:EBF}. Additionally, 
 once bounds for the ratios and the products are available, other expressions involving modified Bessel functions
 can be also considered. For instance, 
 considering the differentiation and recursion formulas for modified Bessel functions, it is easy to see that
$$
\Frac{d}{dx}\log\left(\Frac{I_{\nu}(x)}{K_{\nu}(x)}\right)=\Frac{1}{xI_{\nu}(x)K_{\nu}(x)},
$$
which follows by differentiating and using the Wronskian relation 
$${\cal W}\{K_{\nu}(x),I_{\nu}(x)\}=K_{\nu}(x)I'_\nu (x)-K_{\nu}'(x)I_\nu (x)=1/x.$$

The derivative $\Frac{d}{dx}\log(I_{\nu}(x)/K_{\nu}(x))$ appears in the study of the Casimir effect \cite{Elli:2010:CEA,Lambiase}.
If a bound for the product is used such that the right-hand side is analytically integrable 
(which is the case for the bounds in Theorem \ref{cotprod} with 
$p+q\le 2$ and the bound of Theorem \ref{theorema}) then it is possible to bound the function $I_\nu (b) K_\nu (a)/(I_\nu (b) K_\nu (b))$, which
is called reduction factor in the context of Schwarz methods for solving PDEs \cite{Gigante}. 

Apart from the direct use of the inequalities for the ratios and the products, which 
are in many occasions important for extracting analytical information in equations involving these functions, the results presented in this
paper also have applications as methods for numerical approximations of the functions. To begin with, the bounds obtained, and particularly the
optimized bounds, are accurate enough to give fair approximations for the modified Bessel functions ratios and product
 for moderate $\nu$ and any $x>0$, as shown in Fig. 1. 
These  approximations can be used for moderate accuracy computations of the functions or to accelerate higher accuracy computations.

For example, as discussed in \cite[sect 4.4]{Seg11}, the bounds for the ratios of the modified Bessel function of the first kind
 can be used as seed values for the evaluation of the backward evaluation of the continued fraction representation for $I_{\nu-1}(x)/I_{\nu}(x)$,
 which is no other thing that the iteration of the first relation in (\ref{CF}). The iteration of (\ref{CF}), being $I_\nu (x)$ minimal 
 solution of the corresponding three-term recurrence relation, is a convergent
 process which does not improve the sharpness of the approximation as $x\rightarrow +\infty$ or as $\nu\rightarrow +\infty$ with $x/\nu$ fixed and,
 for this reason, using the optimized bounds as starting values in this iteration is advantageous for accelerating the convergence in those limits. 
 This was already true for the more elementary bounds, and it is even truer for the higher accuracy optimized bounds.
 
 Finally, it is worth mentioning that, although the results in this paper have been 
 proved to be bounds for modified Bessel functions of real order and
 variable, it is expected, by analytic continuation, that they keep providing accurate values for the functions 
 in the complex plane in some regions (away from the zeros of the functions). 
 Figure 2 illustrates this fact for the particular case of the product $I_{15}(z)K_{15}(z)$, showing
  the relative deviation of the bounds $L_{25}^{(+\infty)}(x+i y)$ (left) and $U_{25}^{(+\infty)}(x+i y)$ (right) with
 respect to $I_{25}(x+iy)K_{25}(x+iy)$ in absolute value. 
 The bounds loose accuracy close to the imaginary axis above the turning point $z=i\nu$, where the functions oscillate and have zeros, and
 in a region surrounding the turning point.
 The same behavior is observed for the bounds of the ratios and, as a consequence, the same procedure 
 considered for accelerating the convergence of the continued fractions for real variable seems also possible for complex variable, 
 in particular starting the
 process from sufficiently large orders (away from turning points and oscillatory regions).

\begin{figure}[htbp]
\begin{center}
\begin{minipage}{6cm}
 \includegraphics[width=1\textwidth]{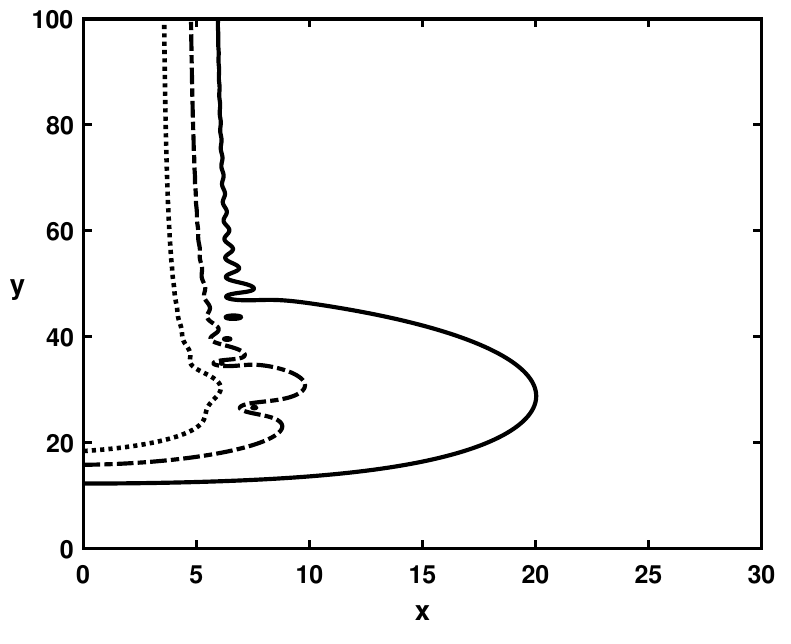}
\end{minipage} 
\begin{minipage}{6cm}
 \includegraphics[width=1\textwidth]{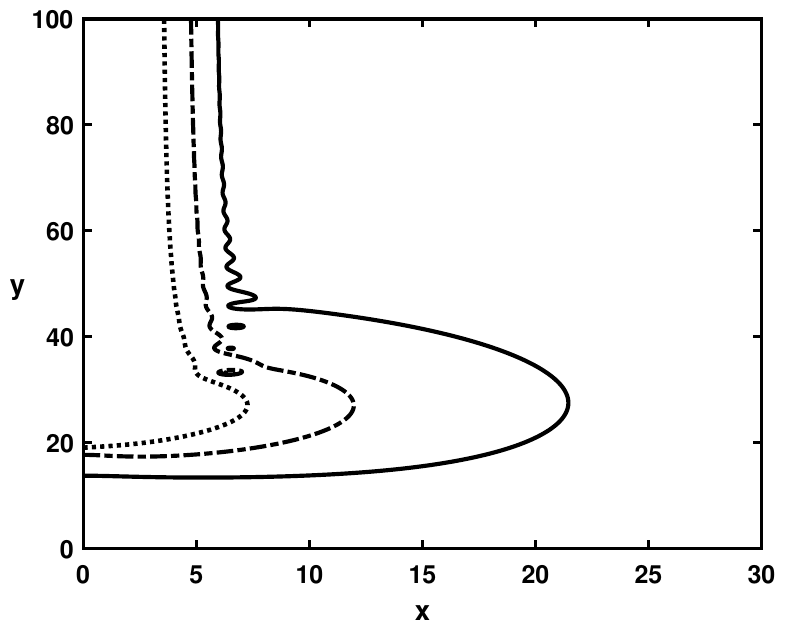}
\end{minipage} 
\end{center}
 \caption{Curves where the relative deviation of the bounds $L_{25}^{(+\infty)}(x+i y)$ (left) and $U_{25}^{(+\infty)}(x+i y)$ (right) with
 respect to $I_{25}(x+iy)K_{25}(x+iy)$ is, in absolute value, $10^{-3}$ (dotted line), $10^{-4}$ (dashed line) and $10^{-5}$ (solid line)}
\end{figure}

 \section*{Declarations}

\subsection*{Author contribution declaration}

J. Segura was initially the sole author of the paper and wrote the manuscript. After the first revision, S. Suzuki shared with J. Segura a proof of Lemma 5, and he was subsequently added as a co-author. Both authors approved the final manuscript.

 \subsection*{Availability of data and materials}
 
 Not applicable.
 
 \subsection*{Funding}

 J. Segura acknowledges support from Ministerio de Ciencia, Innovaci\'on y Universidades, project 
 PID2024-159583NB-I00 (MCIN/AEI/10.13039/501100011033/ FEDER, UE).
 
 \subsection*{Competing interests}
 
 The authors have no competing interests.
 
 \subsection*{Acknowledgements}
 
 J. Segura thanks the editors for his help, particularly for allowing that an additional co-author was included in the final revision and that the paper was expanded by adding a full proof of a previous conjecture. The authors thank the reviewers for their useful comments and suggestions.


\begin{thebibliography}{10}



\bibitem{Baricz:2016:BFT}
\'A. Baricz, D.  Jankov Ma\v sirevi\'c, S. Ponnusamy, S. Singh.
\newblock
Bounds for the product of modified {B}essel functions
\newblock
{\em Aequationes Math.}, 90(4): 859--870, 2016.

\bibitem{Elli:2010:CEA}
S. \AA{}. Ellingsen, I. Brevik, K. A. Milton.
\newblock
Casimir effect at nonzero temperature for wedges and cylinders.
\newblock
{\em Phys. Rev. D}, 81(6): 065031, 2010.

\bibitem{Fucci:2024:VEO}
G. Fucci. C. Romaniega.
\newblock
Vacuum energy of scalar fields on spherical shells with general matching conditions.
\newblock
{\em J. Phys. A: Math. Theor.}, 57 095401, 2024.

\bibitem{Gaunt:2014}
R. E. Gaunt.
\newblock
Inequalities for modified Bessel functions and their integrals.
\newblock
{\em J. Math. Anal. Appl.} 420(1): 373--386, 2014.

\bibitem{Gaunt}
R. E. Gaunt.
\newblock
Uniform bounds for expressions involving modified Bessel functions.
\newblock
{\em Math. Inequal. Appl.} 19(3): 1003-1012, 2016.


\bibitem{Gigante}
G. Gigante, M. Pozzoli, C. Vergara.
\newblock
Optimized Schwarz Methods for the Diffusion-Reaction Problem with Cylindrical Interfaces.
\newblock
{\em SIAM J. Numer. Anal.}, 51 (6): 3402-3430, 2013.


\bibitem{Gil07}
A. Gil, J. Segura, N. M. Temme.
\newblock
Numerical methods for special functions-
\newblock 
SIAM (2007), ISBN 978-0-898716-34-4.

\bibitem{Ikoma}
M. Ikoma, S. Suzuki.
\newblock
Optimal constants of smoothing estimates for the 3D Dirac equation. 
\newblock
{\em Anal.Math.Phys.} 15, 90 (2025)

\bibitem{Lambiase}
G. Lambiase, V. V. Nesterenko, M. Bordag
\newblock
Casimir energy of a ball and cylinder in the zeta function technique. 
\newblock
{\em J. Math. Phys.} 40 (12): 6254–6265, 1999

\bibitem{Nas:1978:RBF}
I. N\.asell.
\newblock
Rational bounds for ratios of modified {B}essel functions.
\newblock
{\em SIAM J. Math. Anal.}, 9(1): 1--11, 1978.

            
\bibitem{NIST}
F. W. J. Olver and L. C. Maximon.
\newblock
Bessel functions,
\newblock
 in N{IST} handbook of mathematical functions. Editors: F. W. J. Olver, D. W. Lozier, R. F. Boisvert, C.
W. Clark.
 \newblock 
 Cambridge University Press, Cambridge, 2010


\bibitem{Ruiz:2016:ANT}
D. Ruiz-Antol\'in, J. Segura.
\newblock
A new type of sharp bounds for ratios of modified {B}essel functions
\newblock
{\em J. Math. Anal. Appl.}, 443(2): 1232--1246, 2016.


\bibitem{Seg11}
J. Segura.
\newblock
Bounds for ratios of modified Bessel functions and associated Turán-type inequalities
\newblock
{\em J. Math. Anal. Appl.}, 374(2): 516-528, 2011.

\bibitem{Seg21}
J. Segura.
\newblock
Monotonicity properties for ratios and products of modified
              {B}essel functions and sharp trigonometric bounds.
\newblock
{\em Results Math.}, 76(4): Paper No. 221, 22, 2021.

\bibitem{Seg23}
J. Segura.
\newblock
Simple bounds with best possible accuracy for ratios of modified Bessel functions
\newblock
{\em J. Math. Anal. Appl.}, 526(1): 127211, 2023. 


\bibitem{Truc:2012:EBF}
F. Truc.
\newblock
Eigenvalue bounds for radial magnetic bottles on the disk.
\newblock
{\em Asymptot. Anal.}, 76(3-4): 233--248, 2012.


\end{thebibliography}
\end{document}